\documentclass[12pt, a4paper]{amsart}

\usepackage{amsmath}
\usepackage{amssymb}
\usepackage{amsthm}
\usepackage{geometry}
\usepackage{hyperref}

\newtheorem{theorem}{Theorem}[section]
\newtheorem{lemma}[theorem]{Lemma}
\newtheorem{definition}[theorem]{Definition}
\newtheorem{corollary}[theorem]{Corollary}
\newtheorem{remark}[theorem]{Remark}

\newcommand{\ct}{\mathcal{T}}
\newcommand{\ttn}{\left(T(t)\right)_{t\geq 0}}
\newcommand{\ssn}{\left(S(t)\right)_{t\geq 0}}
\newcommand{\norm}[1]{\left\lVert#1\right\rVert}

\title[ergodic characterization of reflexivifty]{A semigroup analogue of the\\
 Fonf--Lin--Wojtaszczyk ergodic characterization\\
  of reflexive Banach spaces with a basis}
\author{\textsc{Delio Mugnolo} \\ University of T\"ubingen}
\thanks{The author is supported by the Istituto Nazionale di Alta Matematica ``Francesco Severi''.\\[3pt]
This article  was originally published in Studia Mathematica 164 (3) (2004), 243--251.}

\begin{document}

\maketitle

\begin{abstract}
In analogy to a recent result by V. Fonf, M. Lin, and P. Wojtaszczyk, we prove the following characterizations of a Banach space $X$ with a basis.
\begin{itemize}
    \item[(i)] $X$ is finite-dimensional if and only if every bounded, uniformly continuous, mean ergodic semigroup on $X$ is uniformly mean ergodic.
    \item[(ii)] $X$ is reflexive if and only if every bounded strongly continuous semigroup is mean ergodic if and only if every bounded uniformly continuous semigroup on $X$ is mean ergodic.
\end{itemize}
\end{abstract}

\section{Introduction}

E.R. Lorch proved in the 1930s that if a Banach space $X$ is reflexive, then every power-bounded operator $T$ on $X$ is mean ergodic, i.e., the sequence $\left(\frac{1}{n}\sum_{k=1}^n T^kx\right)_{n\in\mathbb{N}}$ converges for all $x\in X$. In a recent article Fonf, Lin, and Wojtaszczyk have proven (\cite{FLW01}, Cor.~1) that the converse is also true if we assume $X$ to have a basis. Indeed, under this assumption, namely that $X$ is a non-reflexive Banach space with a basis, they have been able to construct a power-bounded operator $T$ on $X$ which is not mean ergodic.

Paralleling their technique, we show in Section~3 that an analogous characterization holds if bounded strongly continuous semigroups of operators are considered instead of power-bounded operators. To this purpose, we introduce in Section~2 a semigroup that in turn permits to prove an ergodic characterization of finite-dimensional Banach spaces that is the semigroup analogue of \cite{FLW01}, Cor.~3.

We emphasize that mean ergodicity of a bounded strongly continuous semigroup $\ttn$ is \textit{not} equivalent to mean ergodicity of the individual operators $T(t)$, $t>0$. If for some $t_0>0$ the power-bounded operator $T(t_0)$ is mean ergodic, so is the bounded semigroup $\ttn$, cf. \cite{DS58}, Thm.~VII.7.1 for $t_0=1$. However, there exist examples of bounded, mean ergodic, strongly continuous semigroups $\ttn$ so that no operator $T:=T(t_0)$, $t_0>0$, is mean ergodic, cf. \cite{EN00}, Expl.~V.4.13.3 (see also \cite{Kr85}, p.~83).

\section{Introductory results}

We begin with a brief reminder of some ergodic theory, and refer to \cite{Kr85} or \cite{EN00}, \S~V.4 for details.

\begin{definition}
Let $\ct:=\ttn$ be a strongly continuous semigroup of linear operators on a Banach space $X$. We denote by
\[
C(r)x:=\frac{1}{r}\int_0^r T(s)x\,ds,\qquad\qquad x\in X,\enspace r>0,
\]
the Ces\`aro means of $\ct$. Then the semigroup $\ct$ is called mean ergodic (uniformly mean ergodic, resp.), if $(C(r))_{r>0}$ converges strongly (with respect to the operator norm, resp.) as $r\rightarrow\infty$.
\end{definition}

The following useful characterization of mean ergodicity of a strongly continuous semigroup is due to Nagel (\cite{Na73}, Thm.~1.7). It is the semigroup analogue of the result of Sine stated in Lemma~3.5 below.

\begin{lemma}\label{lem:2.2}
A bounded strongly continuous semigroup is mean ergodic if and only if its fixed space separates the fixed space of its adjoint.
\end{lemma}

\begin{remark}\label{rem:2.3}
We recall that the fixed space of a strongly continuous semigroup equals the null space of its generator (use \cite{EN00}, Cor.~IV.3.8(i)). Using the theory of sun dual semigroups, one can also see that the same holds for the adjoint of a strongly continuous semigroup (which in general is \textit{not} strongly continuous). Hence, a bounded strongly continuous semigroup is mean ergodic if and only if the null space of its generator $A$ separates the null space of its adjoint $A'$. In particular, a bounded strongly continuous semigroup is mean ergodic if the adjoint of its generator has trivial null space.
\end{remark}

For uniformly mean ergodic semigroups the following characterization is well-known, cf. \cite{EN00}, Thm.~V.4.10.

\begin{lemma}\label{lem:2.4}
A bounded strongly continuous semigroup is uniformly mean ergodic if and only if either its generator is invertible with bounded inverse, or the origin is a pole for its resolvent operator.
\end{lemma}

For a thorough treatment of the theory of Schauder bases and Schauder decompositions on infinite dimensional spaces we refer the reader to \cite{Si70}, \S~I.14 or \cite{GD92}, Chapt.~1. We only recall that the notions of basis and Schauder basis are equivalent on Banach spaces, cf. \cite{Si70}, pag.~170, and that a Schauder basis is always associated with a Schauder decomposition, but the converse implication fails to hold. Let us moreover mention the following general result: If $X$ is a Banach space with a Schauder decomposition $X=\sum_{k\in \mathbb{N}}X_k$, then the norm of $X$ can, without loss of generality, be assumed to make the operators $Q_h$, $h\in\mathbb{N}$, defined by 
\begin{equation}\label{eq:2.1} \tag{2.1}
Q_h x:=x_h,\qquad P_h x:=\sum_{k=1}^h x_k,\qquad\qquad x:=\sum_{k\in\mathbb{N}}x_k\in X,
\end{equation}
contractive. Observe that, due to the uniqueness of the Schauder expansion of $x\in X$, $x=0$ if and only if $P_n x=0$ for all $n\in \mathbb{N}$. With this notation we now prove the following.

\begin{lemma}\label{lem:2.5}
Let $X$ be a Banach space with a Schauder decomposition. Then
\[
M(t)x:=\sum_{h\in\mathbb{N}}e^{-\frac{t}{h}} Q_h x,\qquad\qquad x\in X,\enspace t\ge0,
\]
defines a bounded uniformly continuous semigroup $\{\mathcal{M}\}:=(M(t))_{t\ge0}$ on $X$.
\end{lemma}

For the proof (and later on) we will repeatedly use the auxiliary family $(b_{n,t})_{n\in \mathbb{N}}$, $t\ge0$, of sequences of positive numbers defined by
\begin{equation}\label{eq:2.2} \tag{2.2}
b_{1,t}:=e^{-t}\quad\textrm{and}\quad b_{n,t}:=e^{-\frac{t}{n}}-e^{-\frac{t}{n-1}}\qquad\qquad\textrm{for}\;\; n=2,3,\ldots,\enspace t\ge0.
\end{equation}
Observe that
\begin{equation}\label{eq:2.3} \tag{2.3}
\sum_{h=m+1}^n b_{h,t} = e^{-\frac{t}{n}} - e^{-\frac{t}{m}}, \qquad \sum_{h=1}^n b_{h,t} = e^{-\frac{t}{n}},\qquad\textrm{and}\quad \sum_{h=m+1}^\infty b_{h,t} = 1 - e^{-\frac{t}{m}}
\end{equation}
for all $m,n\in \mathbb{N}$, $m<n$, $t\ge0$.

\begin{proof}
We first show that $\mathcal{M}$ is well defined on $X$. Let $x=\sum_{k\in \mathbb{N}}x_k\in X$. Then
\[
\sum_{k=m+1}^n e^{-\frac{t}{k}}x_k \stackrel{\eqref{eq:2.3}}{=} \sum_{k=m+1}^n\bigg(\sum_{h=m+1}^k b_{h,t} + e^{-\frac{t}{m}}\bigg)x_k = e^{-\frac{t}{m}}\sum_{k=m+1}^n x_k + \sum_{j=m+1}^n b_{j,t}\bigg(\sum_{h=j}^n x_h\bigg),
\]
for all $m,n\in \mathbb{N}$, $m<n$, and the convergence of $\sum_{k\in \mathbb{N}} x_k$ yields that $\left(\sum_{k=1}^n e^{-\frac{t}{k}}x_k\right)_{n\in \mathbb{N}}$ is a Cauchy sequence. It is also clear that $\mathcal{M}$ is a semigroup, and that its generator $A$, defined by
\[
Ax:=-\sum_{h\in \mathbb{N}}\frac{1}{h}Q_h x,\qquad\qquad x\in X,
\]
is bounded, hence $\mathcal{M}$ is uniformly continuous. Also, $\mathcal{M}$ is bounded since for $t\ge0$ and $m\in \mathbb{N}$
\begin{align*}
\sum_{k=1}^m (1-e^{-\frac{t}{k}})x_k &\stackrel{\eqref{eq:2.3}}{=} \sum_{k=1}^m \sum_{h=k+1}^\infty b_{h,t}x_k = \sum_{h=2}^m b_{h,t}\bigg(\sum_{k=1}^{h-1} x_k\bigg) + \sum_{h=m+1}^\infty b_{h,t}\bigg(\sum_{k=1}^m x_k\bigg)\\
&=\sum_{h=2}^m b_{h,t}P_{h-1}x + \sum_{h=m+1}^\infty b_{h,t}P_m x.
\end{align*}
Recall that the operators $P_h$, $h\in\mathbb{N}$, are contractive, and hence we obtain
\[
\left\lVert\sum_{k=1}^m (1-e^{-\frac{t}{k}})x_k\right\rVert \le\sum_{h=2}^m b_{h,t}\norm{x} +\sum_{h=m+1}^\infty b_{h,t}\norm{x} = \sum_{h=2}^\infty b_{h,t}\norm{x} \stackrel{\eqref{eq:2.3}}{=}(1-e^{-t})\norm{x}.
\]
It follows that
\[
\norm{M(t)x-x} =\left\lVert\sum_{k\in\mathbb{N}}(1-e^{-\frac{t}{k}})x_k\right\rVert\le(1-e^{-t})\norm{x}\le\norm{x},
\]
and therefore $\norm{M(t)-I}\le 1$, hence $\norm{M(t)}\le 2$ for all $t\ge0$.
\end{proof}

\begin{theorem}
Let $X$ be a Banach space with a basis. Then $X$ is finite-dimensional if and only if every bounded, mean ergodic, uniformly continuous semigroup on $X$ is uniformly mean ergodic.
\end{theorem}

\begin{proof}
The necessity follows by \cite{EN00}, Cor.~I.2.11 and Thm.~V.4.10. Assume now $X$ to be infinite-dimensional with a basis, thus in particular with a Schauder decomposition. Hence, we can define, as in Lemma~\ref{lem:2.5}, the bounded uniformly continuous semigroup $\{\mathcal{M}\}$ generated by $A$. We prove that $\mathcal{M}$ is mean, but not \textit{uniformly} mean ergodic.

Let $x'\in \ker(A')$. Then
\[
0=\left<Ax,x'\right>=-\sum_{k\in\mathbb{N}}\frac{\left<Q_k x,x'\right>}{k}
\]
for all $x\in X$. Taking into account the uniqueness of the expansion of elements of $X$ we obtain that $x'=0$. By Remark~\ref{rem:2.3} this implies that $\mathcal{M}$ is mean ergodic.

The generator $A$ is injective, and its inverse is given by
\begin{equation}\label{eq:2.4} \tag{2.4}
A^{-1}x=-\sum_{h\in \mathbb{N}}h\mkern3mu Q_h x
\end{equation}
for those $x\in X$ such that the right-hand side of (2.4) converges. Hence $A^{-1}$ is not bounded. Moreover, $0$ is an accumulation point for the set of the eigenvalues of $A$. The claim now follows by Lemma~2.4. Thus, by Lemma~2.4 the bounded semigroup $\mathcal{M}$ is not uniformly mean ergodic.
\end{proof}

\section{Ergodic characterizations of reflexivity}

The technique involved in the proof of Theorem~3.4 below relies on a general result on the geometry of Banach spaces. Answering a question raised by Singer (\cite{Si62}, P~2), Zippin was able to prove (\cite{Zi68}, Thm.~1) that on every non-reflexive Banach space with a basis there exists a non-shrinking basis (we refer to \cite{Zi68} for details). Using this property, Fonf, Lin, and Wojtaszczyk have obtained (\cite{FLW01}) a result that can be summarized as follows.

\begin{lemma}
Let $X$ be a non-reflexive Banach space with a basis. Then there exists a Schauder decomposition $X=\sum_{k\in \mathbb{N}}X_k$, an (equivalent) norm $\norm{\cdot}$ on $X$, a functional $\mathbf{f}\in X'$, and a sequence $(\mathbf{e}_k)_{k\in\mathbb{N}}\in \prod_{k\in \mathbb{N}} X_k$ such that $\norm{\mathbf{e}_k}\le1$ and $\mathbf{f}(\mathbf{e}_k)=1$ for all $k\in \mathbb{N}$, and the operators $P_h$ and $Q_h$, $h\in \mathbb{N}$, defined as in \eqref{eq:2.1} are contractive with respect to $\norm{\cdot}$.
\end{lemma}

\begin{remark}
The proof (\cite{FLW01}, page 149) shows in particular that the sequence $(\mathbf{e}_k)_{k\in\mathbb{N}}$ is in general \textit{not} a basis; in fact, the space $X_1$ need not be one-dimensional and accordingly the Schauder decomposition $X=\sum_{k\in \mathbb{N}}X_k$ need not be associated with a basis.
\end{remark}

In the following we use the notations of Lemma~2.4 and Lemma~3.1.

\begin{lemma}
Let $X$ be a non-reflexive Banach space with a basis. Define a family of operators on $X$ by
\[
N_tx:=\sum_{h\in\mathbb{N}}\mathbf{f}(P_h x)b_{h+1,t}\mathbf{e}_{h+1},\qquad\qquad x\in X,\;\; t\ge0,
\]
with $(b_{n,t})_{n\in \mathbb{N}}$, $t\ge0$, as in \eqref{eq:2.2}. Then
\[
T(t)x:=M(t)x+N_tx,\qquad\qquad x\in X,\; t\ge0,
\]
defines a bounded uniformly continuous semigroup $\ct:=\ttn$ on $X$.
\end{lemma}

\begin{proof}
Since $M(0)=I$ and $N_0=0$, $T(0)=I$ holds. Let now $x_k\in X_k$, $k\in \mathbb{N}$, $t,s\ge0$. To show that $\ct$ satisfies the semigroup law, it suffices to check that $T(t+s)x_k=T(t)T(s)x_k$, or rather, since $(M(t))_{t\ge0}$ is a semigroup by Lemma~\ref{lem:2.5}, that
\[
M(t)N_s x_k + N_t M(s)x_k + N_t N_s x_k=N_{t+s}x_k.
\]
Observe that
\[
N_t x_k=\mathbf{f}(x_k)\!\sum_{h=k+1}^\infty b_{h,t} \mathbf{e}_h, \quad\textrm{and}\;\; \textrm{in particular}\quad N_t \mathbf{e}_k=\!\sum_{h=k+1}^\infty b_{h,t} \mathbf{e}_h.
\]
Hence, we obtain
\begin{align*}
M(t)N_s x_k + & N_t M(s)x_k + N_t N_s x_k\\
&=\sum_{j=k+1}^\infty \mathbf{f}(x_k)b_{j,s} e^{-\frac{t}{j}} \mathbf{e}_j + e^{-\frac{s}{k}}\!\!\sum_{h=k+1}^\infty \mathbf{f}(x_k)b_{h,t} \mathbf{e}_h + \sum_{j=k+1}^\infty \mathbf{f}(x_k) b_{j,s}\bigg(\sum_{h=j+1}^\infty b_{h,t} \mathbf{e}_h\bigg)\\
&=\mathbf{f}(x_k)\left(\sum_{j=k+1}^\infty\left(e^{-\frac{t}{j}} b_{j,s} + e^{-\frac{s}{k}} b_{j,t} \right)\mathbf{e}_j + \sum_{l=k+2}^\infty b_{l,t} \bigg(\sum_{j=k+1}^{l-1} b_{j,s}\bigg)\mathbf{e}_l\right)\\
&\stackrel{\eqref{eq:2.3}}{=} \mathbf{f}(x_k)\left(\sum_{j=k+1}^\infty \left(e^{-\frac{t}{j}} b_{j,s} + e^{-\frac{s}{k}} b_{j,t} \right)\mathbf{e}_j + \sum_{j=k+1}^\infty b_{j+1,t} \left(e^{-\frac{s}{j}}-e^{-\frac{s}{k}}\right)\mathbf{e}_{j+1}\right)\\
&=\mathbf{f}(x_k)\;\;\sum_{j=k+1}^\infty \left(e^{-\frac{t+s}{j}} - e^{-\frac{t+s}{j-1}}\right)\mathbf{e}_j = \mathbf{f}(x_k) \sum_{j=k+1}^\infty b_{j,t+s} \mathbf{e}_j =N_{t+s}x_k.
\end{align*}
Observe now that $N_t/t$ converges strongly as $t\to 0^+$, and we obtain
\[
\dot{N}x:=\lim_{t\to 0^+}\frac{N_t x}{t}=\sum_{h\in \mathbb{N}} \frac{\mathbf{f}(P_h x)}{h^2+ h}\mathbf{e}_{h+1},\qquad\qquad x\in X.
\]
The semigroup $\ct$ is then generated by the operator $B$ given by
\[
Bx:=Ax+\dot{N}x=-\sum_{h\in \mathbb{N}}\left(\frac{1}{h}Q_h x - \frac{\mathbf{f}(P_h x)}{h^2+h}\mathbf{e}_{h+1}\right),\qquad\qquad x\in X.
\]
Since $\norm{\dot{N}}\le\norm{\mathbf{f}}$, and hence also $B$ is bounded, it follows that $\ct$ is uniformly continuous. Finally, observe that
\[
\norm{N_t x} = \left\lVert \sum_{j\in\mathbb{N}} \mathbf{f}(P_j x)b_{j+1,t} \mathbf{e}_{j+1}\right\rVert\le \norm{\mathbf{f}}\sum_{j\in \mathbb{N}}b_{j+1,t}\norm{x}\stackrel{\eqref{eq:2.3}}{\le}\norm{\mathbf{f}}\;\;\norm{x}\qquad\qquad\textrm{for all}\;\; x\in X.
\]
Here, we have used the fact that the operators $P_h$, $h\in\mathbb{N}$, are contractive by Lemma~3.1. Since the semigroup $(M(t))_{t\ge 0}$ is bounded, it also follows that $\norm{T(t)}\le \norm{M(t)} + \norm{N_t}\le 2+\norm{\mathbf{f}}$ for all $t\ge0$.
\end{proof}

\begin{theorem}
If $X$ is a Banach space with a basis, then the following are equivalent:
\begin{itemize}
    \item[(a)] $X$ is reflexive.
    \item[(b)] Every bounded strongly continuous semigroup on $X$ is mean ergodic.
    \item[(c)] Every bounded uniformly continuous semigroup on $X$ is mean ergodic.
\end{itemize}
\end{theorem}

\begin{proof}
(a)$\Rightarrow$ (b) It is a well-known result due to Lorch, cf. \cite{EN00}, Expl.~V.4.7.

(b)$\Rightarrow$ (c) Obvious.

(c)$\Rightarrow$ (a) In Lemma~3.3 we have constructed a semigroup $\ct$ on a non-reflexive Banach space with a basis. We prove that $\ct$ is not mean ergodic by showing that its generator $B$ has trivial null space, and that $\mathbf{f}\not=0$ is in the null space of $B'$. Let $x\in \ker B$, then $-Ax=\dot{N}x$ and hence
\[
\sum_{k\in\mathbb{N}} \frac{Q_k x}{k}=\sum_{k\in \mathbb{N}} \frac{\mathbf{f}(P_k x)}{k^2+k}\mathbf{e}_{k+1}.
\]
By the uniqueness of the Schauder expansion of elements of $X$ we obtain that $x$ solves the system
\[
\left\lbrace\begin{aligned}
Q_1 x&=0\\
Q_k x&= \frac{\mathbf{f}(P_{k-1}x)}{k-1}\mathbf{e}_k, \qquad k=2,3,\ldots.
\end{aligned}\right.
\]
It suffices to show by induction that $P_n x=0$ for all $n\in \mathbb{N}$, with $P_1x=x_1=0$. Let $P_{n-1}x=0$. It follows that $Q_n x=0$, hence $P_n x=P_{n-1} x + Q_n x = 0$, and we conclude that $\ker(B)=\{0\}$. Let now $x_k\in X_k$, $k\in \mathbb{N}$. Then
\[
Bx_k=-\frac{1}{k}x_k + \mathbf{f}(x_k)\sum_{h=k}^\infty \frac{1}{h^2+h}\mathbf{e}_{h+1},
\]
and we obtain
\[
\left< x_k,B' \mathbf{f}\right> = \left<Bx_k, \mathbf{f}\right>= \mathbf{f}(x_k)\left(-\frac{1}{k}+\sum_{h=k}^\infty\frac{1}{h^2+h}\right)=0.
\]
It follows by linearity that $\mathbf{f}\in\ker(B')$, and by Lemma~\ref{lem:2.2} the claim holds.
\end{proof}

The following alternative proof of Theorem~3.4 is due R. Nagel, and seems to be much shorter, but is based on the original result of Fonf, Lin, and Wojtaszczyk. Recall the following, due to R. Sine (\cite{Se70}).

\begin{lemma}
A power-bounded operator is mean ergodic if and only if its fixed space separates the fixed space of its adjoint.
\end{lemma}

\begin{proof}[Alternative proof of the implication (c)$\Rightarrow$ (a) in Theorem~3.4]
Let $(X,\norm{\cdot})$ be a non-reflexive Banach space with a basis. Then, by \cite{FLW01}, Cor.~1 there exists a power-bounded operator $T$ that is not mean ergodic. Consider now the semigroup $\ssn$ generated by the bounded operator $G:=T-I$ on $X$. Since the null space of the generator $G$ is the fixed space of $T$, the semigroup $\ssn$ is not mean ergodic by Lemma~3.5 and Remark~\ref{rem:2.3}. It remains to prove that $\ssn$ is bounded. Observe now that formula (2.5) in \cite{FLW01} shows that the power-bounded operator $T$ is in general not a contraction. Define a norm on $X$ by
\[
{\left\vert\kern-0.25ex\left\vert\kern-0.25ex\left\vert x
    \right\vert\kern-0.25ex\right\vert\kern-0.25ex\right\vert} := \sup_{n\in \mathbb{N}}\norm{T^n x}.
\]
The norm ${\left\vert\kern-0.25ex\left\vert\kern-0.25ex\left\vert \cdot \right\vert\kern-0.25ex\right\vert\kern-0.25ex\right\vert}$ is equivalent to the original norm $\norm{\cdot}$, and $T$ is contractive with respect to it. Further, we can now estimate the norm of $\ssn$ with respect to ${\left\vert\kern-0.25ex\left\vert\kern-0.25ex\left\vert \cdot \right\vert\kern-0.25ex\right\vert\kern-0.25ex\right\vert}$ by 
\[
{\left\vert\kern-0.25ex\left\vert\kern-0.25ex\left\vert S(t) \right\vert\kern-0.25ex\right\vert\kern-0.25ex\right\vert} = {\left\vert\kern-0.25ex\left\vert\kern-0.25ex\left\vert e^{-t}e^{tT} \right\vert\kern-0.25ex\right\vert\kern-0.25ex\right\vert} \leq e^{-t}\cdot e^{t{\left\vert\kern-0.25ex\left\vert\kern-0.25ex\left\vert T \right\vert\kern-0.25ex\right\vert\kern-0.25ex\right\vert}} = e^{t({\left\vert\kern-0.25ex\left\vert\kern-0.25ex\left\vert T \right\vert\kern-0.25ex\right\vert\kern-0.25ex\right\vert}-1)}\leq 1\qquad\qquad \hbox{for all }t\ge0,
\]
and therefore $\ssn$ is bounded in the original norm.
\end{proof}

As in \cite{FLW01}, Cor.~2, using a result of Pelczynski, we can also derive the following characterization from Theorem~3.4.

\begin{corollary}
An arbitrary Banach space is reflexive if and only if every bounded uniformly continuous semigroup on any of its closed subspaces is mean ergodic.
\end{corollary}

\noindent\textbf{Example 3.7.} To understand better the construction presented in the first proof of Theorem~3.4, let us consider a simple case where the objects introduced in Lemma~3.1 can be written explicitly. Take the Banach space $X=l^1$, which admits the Schauder decomposition $X=\sum_{k\in \mathbb{N}}\textrm{span}(\mathbf{e}_k)$. Here $\mathbf{e}_k$ is the $k$-th vector of the usual basis of $l^1$, i.e., $\mathbf{e}_k:=(\delta_{jk})_{j\in \mathbb{N}}$. Moreover, the functional $\mathbf{f}:=(1,1,\ldots)\in X'=l^\infty$ fulfills the assumptions of Lemma~3.1. The generator $B$ of the bounded, uniformly continuous, non-mean ergodic semigroup $\ct$ constructed as in the proof of Theorem~3.4 is now given by the matrix
\[
\begin{pmatrix}
-1 & \frac{1}{2} & \frac{1}{6} & \dots & \dots & \frac{1}{h^2-h} & \dots \\
0  & -\frac{1}{2}& \frac{1}{6} & \dots & \dots & \dots & \dots \\
0  & 0 & -\frac{1}{3}& \dots & \dots & \frac{1}{h^2-h} & \dots \\
\vdots & & & \ddots & & \vdots & \\
\vdots & & & & \ddots & \vdots & \\
0 & & & & & -\frac{1}{h} & \dots \\
\vdots & & & & & & \ddots
\end{pmatrix}.
\]

\section*{Acknowledgment}
The topic of this paper was suggested to me by Prof. Jerry Goldstein. I wish to thank him and Prof. Gis\`ele Goldstein for their kind hospitality, for many inspiring discussions, and for their constant support during my visit at the University of Memphis in spring 2002.

\bigskip
\noindent
\textsc{Arbeitsbereich Funktionalanalysis\\
Mathematisches Institut der Universit\"at T\"ubingen\\
Auf der Morgenstelle 10\\
D--72076 T\"ubingen, Germany}\\
\textit{E-mail address}: demu@fa.uni-tuebingen.de

\end{document}